\documentclass[12pt,a4paper]{article}
\usepackage{amsmath}
\usepackage{amsfonts}
\usepackage{amsthm}
\usepackage{amscd}
\usepackage{amsbsy}
\usepackage{amsrefs}
\usepackage[latin2]{inputenc}
\usepackage{t1enc}
\usepackage[mathscr]{eucal}
\usepackage{indentfirst}
\usepackage{graphicx}
\usepackage{graphics}
\usepackage{pict2e}
\usepackage{epic}
\numberwithin{equation}{section}
\usepackage[margin=2.9cm]{geometry}
\usepackage{epstopdf}
\usepackage{tikz-cd}
\usepackage{chngcntr}
\usepackage{amssymb}
\usepackage{esint}
\usepackage{hyperref}
\usepackage{mwe}
\usepackage{mathrsfs}
\usepackage{xfrac}
\usepackage{verbatim}
\usepackage{mathpazo}

\theoremstyle{plain}
\newtheorem{Th}{Theorem}[section]

\newtheorem{Prop}[Th]{Proposition}

 \theoremstyle{definition}
\newtheorem{Def}[Th]{Definition}

\newtheorem{Rem}[Th]{Remark}
\newtheorem{?}[Th]{Problem}

\newcommand*\R{\mathbb{R}}

\newcommand*\Om{\Omega}

\newcommand{\Div}{\text{div}_x}
\newcommand{\Divh}{\text{div}_{x_h}}
\newcommand{\vectoru}{\mathbf{u}}

\newcommand{\vectorv}{\mathbf{v}}

\newcommand{\vectorm}{\mathbf{m}}
\newcommand{\Epsilon}{\mathcal{E}}
\newcommand{\dx}{\text{ d}x}
\newcommand{\dt}{\text{ d}t\;}

\newcommand{\vectorphi}{\pmb{\varphi}}

\newcommand{\curlh}{\text{curl}_{x_h}}
\begin{document}

\title{Multiple scales and singular limits of perfect fluids}

\author{Nilasis Chaudhuri
	\thanks{E-mail:\tt chaudhuri@math.tu-berlin.de}
}
\maketitle

\centerline{Technische Universit{\"a}t, Berlin}
\centerline{Institute f{\"u}r Mathematik, Stra\ss e des 17. Juni 136, D -- 10623 Berlin, Germany.}

\begin{abstract}
In this article our goal is to study the singular limits for a scaled barotropic Euler system modelling a rotating, compressible and inviscid fluid, where Mach number $=\epsilon^m $, Rossby number $=\epsilon $ and Froude number $=\epsilon^n $ are proportional to a small parameter $\epsilon\rightarrow 0$. The fluid is confined to an infinite slab, the limit behaviour is identified as the incompressible Euler system. For \emph{well--prepared} initial data, the convergence is shown on the life span time interval of the strong solutions of the target system, whereas a class of generalized \emph{dissipative solutions} is considered for the primitive system. The technique can be adapted to the compressible Navier--Stokes system in the subcritical range of the adiabatic exponent $\gamma$ with $1<\gamma\leq\frac{3}{2}$, where the weak solutions are not known to exist.

\end{abstract}



\section{Introduction}
We study models of rotating fluids as described in {Chemin et.al.} \cite{CDGG2006}. Let $T>0$ and $\Omega (\subset \mathbb{R}^3) = \mathbb{R}^2 \times (0,1)$ be an infinite slab. We consider the scaled {\emph{compressible Euler equation}} in time-space cylinder $Q_T=(0,T)\times \Omega$ describing the time evolution of the mass density $\varrho=\varrho(t,x)$ and the momentum field $\vectorm=\vectorm(t,x)$ of a rotating inviscid fluid with axis of rotation $\mathbf{b}=(0,0,1)$:

\begin{itemize}
	\item \textbf{Conservation of Mass: }
	\begin{align}
	\partial_t \varrho + \text{div}_x \textbf{m} &=0. \label{cee:cont}
	\end{align}
	\item \textbf{Conservation of Momentum:}
	\begin{align}
	\partial_t \textbf{m} + \Div \left( \frac{\textbf{m} \otimes \textbf{m} }{\varrho} \right) +\frac{1}{\text{Ma}^2}\nabla_x p(\varrho)+\frac{1}{\text{Ro}} \mathbf{b} \times \textbf{m} = {\frac{1}{\text{Fr}^2}\varrho \nabla_{x}G.} \label{cee:mom:sc}
	\end{align}

\end{itemize}
\begin{itemize}
	\item The scaled system contains characteristic numbers:
	\subitem Ma-- Mach number,
	\subitem Ro-- Rossby number,
	\subitem Fr-- Froude number.\\
	Here  we consider,
	\begin{align}\label{scaling}
	\text{Ma} \approx \epsilon^m,\; \text{Ro} \approx \epsilon,\; \text{Fr} \approx \epsilon^n \text{ for } \epsilon>0,\; m,n > 0.
	\end{align}
	\item \textbf{Pressure Law:} The pressure $p$ and the density $\varrho$ of the fluid are interrelated by
	the standard isentropic law
		\begin{align}\label{p-condition}
		\begin{split}
		p(\varrho)=a \varrho^\gamma,\; a>0,\; \gamma> 1.
		\end{split}
		\end{align}
	
 	\item \textbf{Boundary condition:} Here we consider slip condition  on the horizontal boundary, i.e.
 	\begin{align}\label{BC}
 	\vectorm \cdot \mathbf{n}=0,\; \mathbf{n}=(0,0,\pm 1).
 	\end{align}
 	
 	\item \textbf{Far field condition:} Let us introduce the notation $x=(x_h,x_3)$ and $P_h(x)=x_h$. For each $\epsilon > 0$, we identify a static solution that satisfies \eqref{cee:cont}-\eqref{cee:mom:sc} with \eqref{scaling}. More 
	specifically, a static solutions is a pair $(\tilde{\varrho}_{\epsilon},\mathbf{0})$, where the density 
	profile $\tilde{\varrho}_{\epsilon}$ satisfies
 	\begin{align}\label{static-solution}
 	\nabla_{x} p(\tilde{\varrho}_\epsilon)=  \epsilon^{2(m-n)}\tilde{\varrho}_\epsilon\nabla_{x}G.
 	\end{align}
	In general, there are infinitely many static solutions for a given potential $G$.
 	 We assume the \emph{far field} condition as,
 	\begin{align}\label{far_field}
 	|\varrho - \tilde{\varrho}_\epsilon| \rightarrow 0, \; \vectorm \rightarrow \mathbf{0} \text{ as } \vert x_h \vert \rightarrow \infty.
 	\end{align}
 	\item \textbf{Initial data:}
 	For each $\epsilon>0$, we supplement the initial data as
 	\begin{align}\label{initial condition}
 	{\varrho(0,\cdot)=\varrho_{\epsilon, 0},\; \vectorm(0,\cdot)= \vectorm_{\epsilon, 0}.}
 	\end{align}
 	\item \textbf{Choice of $G$:}  As a matter of fact, the driving potential $G$ can be seen as a sum of the centrifugal force proportional to the norm of the horizontal component of the spatial variable i.e. $(x_1^2 + x_2^2)$  and the gravitational force acting in the vertical direction $x_3$. We omit the effect of the centrifugal force in the present paper motivated by certain meteorological models. Instead we consider
 	\begin{align}
 	G(x)=-x_3\text{ in } \Om 
 	\end{align}  
corresponding to the gravitational force acting in the vertical direction.
\end{itemize}

We consider singular limit problem for $\epsilon\rightarrow 0$ in the multiscale regime:
\begin{align}\label{multiscale}
\frac{m}{2}>n\geq1.
\end{align}
Thus we study the effect of \emph{low Mach number limit} (also called \emph{incompressible limit}), \emph{low Rossby number limit} and \emph{low Froude number limit} acting \emph{simultaneously} on the system \eqref{cee:cont}-\eqref{cee:mom:sc}.

Formally, we observe that low Mach number limit regime indicates the fluid becomes \emph{incompressible} and low Rossby number limit indicates fast rotation of the fluid and as a consequence of that fluid becomes \emph{planner} (two-dimensional). \par
As solutions of the (primitive) compressible Euler systems are expected to develop singularities (shock waves) in a finite time, 
there are two approaches to deal with the singular limit problem.
\begin{itemize}
	\item[I.]  The first approach consists of considering  \emph{classical(strong)} solutions of the \emph{primitive system} and expecting it converges to the classical solutions of the \emph{target system}. Here, the main and highly non-trivial issue is to ensure that the lifespan of the strong solutions is bounded below away from zero uniformly with
	respect to the singular parameter.	
	\item[II.] The second approach is based on the concept of \emph{weak},  \emph{measure--valued} or  \emph{dissipative} solutions  of the \emph{primitive system}. Under proper choice of initial data one can show convergence 
	provided the \emph{target system} admits smooth solution.
\end{itemize}

For the first approach in the low Mach number limits we have results by  Ebin \cite{E1977}, Kleinermann and Majda \cite{KM1981}, Schochet \cite{S1986}, and many others. For rotating fluids there are results by Babin, Mahalov and Nicolaenko \cite{BMN1999, BMN2001} and Chemin et. al. \cite{CDGG2006}.\par

In the case of second approach, most of the results dealing with weak solutions have been studied for compressible Navier--Stokes system with additional consideration of high Reynolds number limit. For rotating fluids there are several results, see Feireisl, Gallagher and Novotn{\'y}\cite{FGN2012}, Feireisl et. al. \cite{FGGvN2012}, Feireisl and Novotn{\'y} \cite{FN2014,FN2014ii} and Li\cite{L2019}.\par
Since existence of global-in-time weak solution of compressible Euler equation satisfying energy inequality is still open for general initial data. Hence it is important to consider \emph{measure--valued solution} or newly developed \emph{dissipative solution} for this system. The concept of measure--valued solutions has been studied in variuos context, like, analysis of numerical schemes etc. In the following articles by Alibert and Bouchitt{\'e} \cite{AB1997}, Gwiazda, {\'S}wierczewska-Gwaizda and Wiedemann \cite{PAW2015},  B\v r{e}zina and Feireisl \cite{BF2018b}, B\v rezina \cite{B2018}, Basari\'c \cite{Bd2019}, Feireisl and Luk\' a\v cov\'a-Medvidov\'a \cite{FL2018} we observe the development of theory on measure valued solution for different models describing compressible fluids mainly with the help of Young measures.\par

 Recently, Feireisl, Luk\' a\v cov\'a-Medvidov\'a and Mizerov\'a in \cite{FLM2019} and Breit, Feireisl and Hofmanov\'a \cite{BeFH2019} give a new definition for compressible Euler system, termed as \emph{dissipative solution} without involving Young measures.\par
 
The advantages to consider the second approach are,
\begin{itemize}
	\item \emph{Weak} or \emph{measure valued} solutions to the \emph{primitive system} exist globally in time. Hence the result depends only on the life span of the \emph{target problem} that may be finite.
	\item The convergence holds for a large class of generalized solutions which indicates certain stability of the limit solution of the target system. 
\end{itemize}

In particular, the results involving generalized solutions are better in the sense that convergence holds for a larger class of solutions and on the life span of the limit system. \par

There is a series of works dealing with the low Mach number limit in the framework of measure--valued solutions.
In  Feireisl, Klingenberg and Markfelder \cite{FKM2019}, Bruell and Feireisl \cite{BrF2018}, B\v r{e}zina and M\' acha \cite{BM2018}, it is shown that \emph{measure--valued solution} of \emph{primitive system} which describes some compressible inviscid fluid converges to \emph{strong solution} of incompressible \emph{target system} under consideration of suitable initial data. The `single--scale' limit of our system i.e. $m=1$ and $G=0$, has been studied by Ne{\v c}asov{\'a} and Tong in \cite{NT2018}, again with the help of measure--valued solution. \par

The framework of measure--valued solutions can be applied also in the context of the Navier--Stokes system. 
Although \emph{weak solutions} are available here, their existence is constrained by the technical condition for the 
adiabatic exponent $\gamma>\frac{3}{2}$. To handle this technical restriction,  Feireisl et. al. \cite{FPAW2016} introduced the concept of \emph{dissipative measure-valued solution} in terms of the Young measure. Here we use a slightly 
different approach introducing \emph{dissipative solution} for Navier--Stokes system without an explicit presence of the Young measure. In such a way, we extend the convergence result to the Navier--Stokes system with high Reynolds number limit in the 
regime where the existence of weak solutions is not known. \par
In our approach, it is very important to consider proper initial data mainly termed as \emph{well-prepared} and \emph{ill-prepared} initial data. Feireisl and Novotn{\'y} in \cite{FN2009b},  explain that for \emph{ill-prepared data} the presence of Rossby-acoustic waves play an important role in analysis of singular limits. Meanwhile this effect was absent in \emph{well-prepared} data. Here we deal with the \emph{well-prepared} initial data.\par

Our main goal is to prove that under suitable choice of initial data a \emph{dissipative solution} of compressible rotating Euler system in low Mach and low Rossby regime converges to strong solution of incompressible Euler system in 2D. Hence our plan for the article is,
\begin{itemize}
	\item[1.] Derivation of limit system.
	\item[2.] Definition of dissipative solution.
	\item[3.] Singular limit for `\emph{well-prepared}' data.
	\item[4.] Extension to Navier--Stokes system.
\end{itemize}

\subsection{Notation:}
\begin{itemize}
	\item To begin, we introduce a function
	$\chi = \chi(\varrho)$ such that
	\begin{align}
	\begin{split}
	\chi(\cdot) \in C_c^{\infty}	(0, \infty),\; 0 \leq \chi \leq 1,\; \chi(\varrho) = 1 \ \mbox{if}\ \frac{1}{2} \leq \varrho \leq 2.
	\end{split}
	\end{align}
	For a function, $H = H(\varrho, \vectoru)$ we set
	\begin{align}
	\begin{split}
	[H]_{\text{ess}}= \chi(\varrho) H(\varrho, \vectoru),\; [H]_{\text{res}}= (1-\chi(\varrho)) H(\varrho, \vectoru).
	\end{split}
	\end{align}	
	\item Without loss of generality, we assume the `normalized' setting for $p$ as
	\begin{align}
	p^{\prime}(1)=1.
	\end{align} Let us define pressure potential as,
	\begin{align}\label{Pressure_potential}
	P(\varrho)=\varrho \int_{1}^{\varrho} \frac{p(z)}{z^2} \; \text{d}z.
	\end{align}
	As a consequence of that we have,
	\begin{align}
	\varrho P^{\prime}(\varrho) -P(\varrho)= p(\varrho) \text{ and } \varrho P^{\prime \prime} (\varrho)=p^{\prime}(\varrho) \text{ for } \varrho>0.
	\end{align}
\end{itemize}

\section{Derivation of Limit systems}
Here is an informal justification how we obtain the \emph{target system}.
First we note that $(\tilde{\varrho}_{\epsilon},0)$ is a \emph{steady state} solution for \eqref{cee:cont}-\eqref{cee:mom:sc}.

Let us consider
\begin{align*}
&\varrho_{\epsilon}= \tilde{\varrho}_{\epsilon} + \epsilon^m \varrho_{\epsilon}^{(1)}+ \epsilon^{2m}  \varrho_{\epsilon}^{(2)}+ \cdots,\\
&\vectorm_\epsilon= \tilde{\varrho}_{\epsilon}\vectorv + \epsilon^{m}  \vectorm_\epsilon^{(1)} + \epsilon^{2m} \vectorm_\epsilon^{(2)} + \cdots.
\end{align*}

As a consequence of the above we obtain,
\begin{align*}
p(\varrho_{\epsilon})= p(\tilde{\varrho}_\epsilon) + \epsilon^m p^{\prime}(\tilde{\varrho}_{\epsilon}) \varrho_{\epsilon}^{(1)} + \epsilon^{2m} (  p^{\prime}(\tilde{\varrho}_{\epsilon})\varrho_{\epsilon}^{(2)}+  \frac{1}{2}p^{\prime\prime}(\tilde{\varrho}_{\epsilon})(\varrho_{\epsilon}^{(1)})^{2}) + o(\epsilon^{3m}).
\end{align*}
We have static solution $(\tilde{\varrho}_\epsilon,\mathbf{0})$ 
\begin{align}
\nabla_{x} p(\tilde{\varrho}_\epsilon)= \epsilon^{2(m-n)}\tilde{\varrho}_\epsilon\nabla_{x}G.
\end{align}

Clearly condition on $m\text{ and }n$ in \eqref{multiscale} indicate $\lim_{\epsilon\rightarrow0}\nabla_{x} P^\prime(\tilde{\varrho}_\epsilon)=0$.

Without loss of generality we assume, 
\begin{align*}
\tilde{\varrho}_\epsilon \approx \bar{\varrho} + \epsilon^{2(m-n)}.
\end{align*}

So we obtain,
\begin{align*}
\bar{\varrho}\Div \vectorv + \epsilon^m (\partial_{t} \varrho_{\epsilon}^{(1)}+ \Div (\vectorm_\epsilon^{(1)}) + o(\epsilon^{2m})=0
\end{align*}
and
\begin{align}\label{case3-central}
\begin{split}
&\bar{\varrho}(\partial_{t} \vectorv + (\vectorv\cdot \nabla_x)\vectorv) + \nabla_x (  p^{\prime}(\bar{\varrho})\varrho_{\epsilon}^{(2)}+  \frac{1}{2}p^{\prime\prime}(\bar{\varrho})(\varrho_{\epsilon}^{(1)})^{2}) + \epsilon^{m-1} \mathbf{b} \times \vectorm_\epsilon^{(1)}\\
&+ \frac{1}{\epsilon^m}p^{\prime}(\bar{\varrho}) \nabla_x \varrho_{\epsilon}^{(1)}+\frac{1}{\epsilon} ( \bar{\varrho}\mathbf{b} \times \vectorv) + o(\epsilon)=0.
\end{split}
\end{align}

Let $\mathbb{H}$ be the Helmontz projection, then  we have,
\begin{align}
\mathbb{H}\bigg(\partial_t(\varrho_\epsilon \mathbf{u}_\epsilon) + \Div \bigg( \frac{ \mathbf{m}_\epsilon  \otimes \mathbf{m}_\epsilon}{\varrho_\epsilon} \bigg)+\frac{1}{\epsilon } \mathbf{b} \times \varrho_\epsilon \mathbf{u}_\epsilon \bigg)=\mathbb{H}\bigg( \frac{1}{\epsilon ^{2n}}\varrho_\epsilon  \nabla_{x}G\bigg). 
\end{align}	
Assuming $ \vectorm_{\epsilon} \rightarrow \bar{\varrho} \vectorv$  in some strong sense, multiplying the above equation by $\epsilon$ and using our standard expansion technique, we obtain,
\begin{align*}
\mathbb{H}[\mathbf{b}\times \bar{\varrho}\vectorv]=0
\end{align*}
From above relations we get,
\begin{align*}
&\mathbf{b}\times \vectorv= \nabla_{x} \psi,\\
&\psi_{x_3}=0,\; \psi(x)=q(x_h), \nabla_{x_h}^{\perp}\psi =\vectorv_h, \vectorv_h=(v_1,v_2),\\
&\Divh \vectorv_h=0, v_{3_{x_3}}=0, \text{ with } \nabla_{x_h}^{\perp}\equiv (-\frac{\partial}{\partial x_2},\frac{\partial}{\partial x_1}).
\end{align*}
Thus we have,
\begin{align*}
v_{1_{x_3}}=0,\; v_{2_{x_3}}=0,\; \vectorv(x)=\vectorv(x_h).
\end{align*}
Also boundary condition will lead us to conclude $v_3(x_h,x_3)=0. $
Thus we have $ \vectorv=(\vectorv_h(x_h),0)$. Finally we obtain,
\begin{align}\label{lim_sys_caseIII}
\begin{split}
&\Divh \vectorv_h =0 ,\\ 
&\partial_{t} \vectorv_h + (\vectorv_h \cdot \nabla_{x_h})\vectorv_h + \nabla_{x_h} \Pi=0.
\end{split}
\end{align}
The above system is 2D Euler equation.\\

\section{Definition of dissipative solution}

\subsection{Choice of Static Solution:} 
As we have noticed during the informal discussion a static solution $(\tilde{\varrho}_\epsilon,\mathbf{0})$ satisfies,
\begin{align*}
&\nabla_{x} p(\tilde{\varrho}_\epsilon)= \epsilon^{2(m-n)}\tilde{\varrho}_\epsilon\nabla_{x}G\\
\implies &\nabla_{x} P^{\prime}(\tilde{\varrho}_\epsilon)= \epsilon^{2(m-n)}\nabla_{x}G.\\
%\implies P^{\prime}(\tilde{\varrho}_{\epsilon})= \epsilon^{2(m-n)} G + P^{\prime} 
\end{align*}
As a consequence of $G=(0,0,-x_3)$, we have $ \tilde{\varrho}_\epsilon(x)= \tilde{\varrho}_\epsilon(x_3).$
With an extra assumption $\tilde{\varrho}_\epsilon(0)=1$ we obtain
\begin{align}
P^{\prime}(\tilde{\varrho}_\epsilon)= -\epsilon^{2(m-n)} x_3 + P^{\prime}(1).
\end{align}
Finally we choose \emph{static solution} $\tilde{\varrho}_\epsilon$ with the property,

\begin{align}\label{staic:soln:final}
\sup_{  x_3\in [0,1]} \vert \tilde{\varrho}_\epsilon(x_3)-1 \vert \leq \epsilon^{2(m-n)},\; \sup_{  x_3\in [0,1]} \vert \nabla_{x} \tilde{\varrho}_\epsilon(x_3) \vert \leq \epsilon^{2(m-n)}.
\end{align}
As $m > n$, asymptotically, the static solution approaches the constant state $\tilde{\varrho} = 1$ as 
$\epsilon \to 0$.\par 
Now we will give the definition of \emph{dissipative solution} for the system of our consideration. 
\begin{Def}\label{dmv_defn}
	Let $\epsilon>0$ and $\tilde{\varrho}_\epsilon>0$. We say functions $\varrho_{\epsilon},\vectoru_{\epsilon}$  with,
	\begin{align}
	\varrho_{\epsilon}-\tilde{\varrho}_\epsilon \in  &C_{\text{weak}}([0,T];L^2+L^{\gamma}(\Om)),\;\varrho_\epsilon \geq 0,\; \vectorm_{\epsilon} \in   C_{\text{weak}}([0,T];L^2+L^{\frac{2\gamma}{\gamma+1}}(\Om)),
	\end{align}
	are a \emph{dissipative solution} to the compressible Euler equation \eqref{cee:cont}-\eqref{far_field} with initial data $\varrho_{0,\epsilon}, (\varrho \vectoru)_{0,\epsilon}$ satisfying,
	\begin{align}\label{fe_id}
	\begin{split}
	\varrho_{0,{\epsilon}}\geq 0,\; E_{0,{\epsilon}}=\int_{ \Om} \bigg( \frac{1}{2} \frac{ \vert \vectorm_{0,\epsilon} \vert^2}{\varrho_{0,\epsilon}} + P(\varrho_{0,{\epsilon}})-(\varrho_{0,{\epsilon}}-\tilde{\varrho}_\epsilon)P^{\prime}(\tilde{\varrho}_\epsilon)-P(\tilde{\varrho}_\epsilon)\bigg)\dx < \infty,
	\end{split}
	\end{align}
	if there exist the \emph{turbulent defect measures}
	\begin{align}
	\begin{split}
	&{\mathfrak{R}_{m_{\epsilon}} \in L^{\infty}(0,T;\mathcal{M}^+({\overline{\Om};\mathbb{R}^{d\times d}_{\text{sym}}}))},\; \mathfrak{R}_{e_{\epsilon}} \in L^{\infty}(0,T;\mathcal{M}^{+}(\overline{\Om})),
	\end{split}
	\end{align}
	satisfying compatibility condition
	\begin{align}\label{compat_turb_mes}
	\lambda_1 \text{trace}\mathfrak{R}_{m_{\epsilon}} \leq \mathfrak{R}_{e_{\epsilon}} \leq \lambda_2 \text{trace} \mathfrak{R}_{m_{\epsilon}},\; \lambda_1,\lambda_2>0,
	\end{align}
	such that the following holds,
	\begin{itemize}
		\item \textbf{Equation of Continuity:}
		For any $\tau \in (0,T) $ and any $\varphi \in C_{c}^{1}([0,T)\times \bar{\Om})$ {it holds}
		\begin{align}\label{continuity_eqn_diss}
		\begin{split}
		&\big[ \int_{ \Om}{ \varrho_{\epsilon}} \varphi \dx\big]_{t=0}^{t=\tau}=
		\int_0^{\tau} \int_{\Om} [ \varrho_{\epsilon} \partial_t \varphi +  \vectorm_{\epsilon} \cdot \nabla_x \varphi] \dx \dt .
		\end{split}
		\end{align}
		\item \textbf{Momentum equation:} For any $\tau\in (0,T)$ and any $\pmb{\varphi} \in C^{1}_c([0,T)\times \Om;\mathbb{R}^d)$ with $\pmb{\varphi} \cdot \mathbf{n}|_{\partial\Om}=0,$ it holds
		\begin{align}\label{momentum_eqn_diss}
		\begin{split}
		&\bigg[\int_{\Om}  \vectorm_{\epsilon}(\tau,\cdot)\cdot \vectorphi(\tau,\cdot) \dx\bigg]_{t=0}^{t=\tau} \\
		&=\int_0^{\tau}\int_{\Om} \bigg[ \vectorm_{\epsilon}\cdot \partial_{t} \vectorphi + \bigg(\frac{ \vectorm_{\epsilon} \otimes \vectorm_{\epsilon}}{\varrho_{\epsilon}}\bigg) : \nabla_x \vectorphi + \frac{1}{\epsilon^{2m}}p(\varrho_{\epsilon}) \Div \vectorphi + \frac{1}{\epsilon}\mathbf{b} \times \vectorm_{\epsilon} \cdot \vectorphi \bigg] \dx \dt \\
		&+\int_0^{\tau}\int_{\Om}  \frac{1}{\epsilon^{2n}} \varrho_{\epsilon} \nabla_{x}G\cdot \vectorphi \dx \dt+\int_0^{\tau}\int_{\overline{\Om}} \nabla_x \vectorphi : \text{d}\mathfrak{R}_{m_{\epsilon}}   \dt .
		\end{split}
		\end{align}
		\item \textbf{Energy inequality:}  The total energy $E$ is defined in $[0,T)$ as,
		\begin{align*}
		E_{\epsilon}(\tau)= \int_{\Om} &\bigg( \frac{1}{2} \frac{ {\vert \vectorm_{\epsilon} \vert^2}}{\varrho_\epsilon} +\frac{1}{\epsilon^{2m}}(P(\varrho_{\epsilon})-(\varrho_{\epsilon}-\tilde{\varrho}_\epsilon)P^{\prime}(\tilde{\varrho}_{\epsilon})) -P(\tilde{\varrho}_{\epsilon})\bigg)(\tau ,\cdot)\dx.
		\end{align*}
			It satisfies,
			\begin{align}\label{energy_inequality}
			E_\epsilon (\tau) + \int_{\overline{\Omega}}  {\rm d \; } \mathfrak{R}_{e_\epsilon} (\tau, \cdot) \leq E_{0,\epsilon}
			\end{align}
{for a.a. $\tau > 0$.}		
	\end{itemize}
\end{Def}
\begin{Rem}
	It is important to define the function
	\[ (\varrho_\epsilon,\vectorm_{\epsilon})\mapsto \frac{1}{2} \frac{\vert \vectorm_\epsilon \vert^2}{\varrho_{\epsilon}} \] 
	on the vacuum set as
	\[ (\varrho_{{\epsilon}},\vectorm_\epsilon )\mapsto \frac{1}{2}\frac{\vert \vectorm_\epsilon\vert^2}{\varrho_\epsilon} =
	\begin{cases} 
	\frac{1}{2}\frac{\vert \vectorm_\epsilon \vert^2}{\varrho_\epsilon}  & \text{ if } ,\varrho_\epsilon \neq 0, \vectorm_\epsilon \neq 0 ,\\
	0 & \text{ if } \varrho_\epsilon =0,\; \vectorm_\epsilon=0 ,\\
	\infty & \text{ if }\varrho_\epsilon=0,\; \vectorm_\epsilon\neq 0.
	\end{cases}
	\]
	It follows from the energy inequality \eqref{energy_inequality} that for each $\epsilon>0$,
	\[ \mathcal{L}^4 \bigg( \bigg\{(t,x)\in (0,T)\times \Om \bigg| \frac{1}{2}\frac{\vert \vectorm_\epsilon \vert^2}{\varrho_\epsilon}=\infty   \bigg\}\bigg) =0 ,\]
	where $\mathcal{L}^4$ is the Lebesgue measure in $\R^4$.
\end{Rem}

\begin{Th}
	Suppose $\Om$ be the domain specified above and pressure follows \eqref{p-condition}. If $(\varrho_{0,\epsilon},\vectorm_{0,\epsilon})$ satisfies \eqref{fe_id}, then there exists \emph{dissipative solution} as defined in definition \eqref{dmv_defn}.
\end{Th}

The proof this theorem follows in similar lines of Breit, Feireisl and Hofmanov{\'a} as in \cite{BeFH2019, BeFH2019(2)}. We have to adopt it for unbounded domain as suggested in Basari{\'c} \cite{Bd2019}.

\section{ Singular limit for "Well-prepared" initial Data }

\subsection{Target System }
Taking motivation from \cite{FN2014ii} we expect the target system as,
\begin{align}\label{target_sys_2D euler}
\begin{split}
&\Divh \vectorv_h =0 , \text{ in } \R^2,\\ 
&\partial_{t} \vectorv_h + (\vectorv_h \cdot \nabla_{x_h})\vectorv_h + \nabla_{x_h} \Pi=0,\text{ in } \R^2.
\end{split}
\end{align}

The result stated below by Kato and Lai in \cite{KL1984} ensures the existence and uniqueness of Euler system in 2D:
\begin{Prop}
	Let
	\begin{align*}
	\vectorv_{0} \in W^{k,2}(\R^2;\R^2),\; k\geq 3,\; \Divh \vectorv_{0}=0
	\end{align*}
	be given. Then the system \eqref{target_sys_2D euler} supplemented with initial data $\vectorv_h(0)= \vectorv_{0}$ admits regular solution $(\vectorv_h,\Pi)$, unique in the class 
	\begin{align}\label{reg_tar_system}
	\begin{split}
	&\vectorv_h \in C([0,T];W^{k,2}(\R^2;\R^2)),\; \partial_{t} \vectorv_h \in C([0,T];W^{k-1,2}(\R^2;\R^2)), \\
	&\; \Pi \in C([0,T];W^{k,2}(\R^2)),  
	\end{split}
	\end{align}
	with $\Divh \vectorv_h=0$.
\end{Prop} 
Taking motivation from \eqref{case3-central} we consider another equation that describes some non-oscillatory part, 
\begin{align}\label{non_osc}
\begin{split}
&\nabla_x q_{\epsilon} +  \epsilon^{m-1 }\mathbf{b} \times \vectorv =0.\\
\end{split}
\end{align}
Also we choose $ q_\epsilon (0,\cdot)= q_{0,\epsilon} $ such that $-\Delta_{x_h} q_{0,\epsilon} = \epsilon^{m-1}\bar{\varrho}  \curlh P_h(\vectorv_0). $

\subsection{ Definition of "Well-prepared Data"}
\begin{Def}
	We say that the set of initial data $\{(\varrho_{0,\epsilon},\vectorm_{0,\epsilon})\}_{\epsilon>0} $ is "well-prepared" if,
	\begin{align}\label{well_prepared_id}
	\begin{split}
	&\varrho_\epsilon(0,\cdot)= \varrho_{0,{\epsilon}}= \tilde{\varrho_\epsilon} + \epsilon^m {\varrho}_{0,\epsilon}^{(1)} \text{ with } {\varrho}_{0,\epsilon}^{(1)} \rightarrow 0 \text{ in } L^2(\Om),\\
	&\vectoru_{0,\epsilon}= \frac{\vectorm_{0,\epsilon}}{\varrho_{0,{\epsilon}}} \rightarrow \vectorv_{0}=\big(\vectorv_{0}^{(1)},\vectorv_{0}^{(2)},0\big ) \text{ in } L^2(\Om;\R^3) \text{ with } \Divh \vectorv_{0}=0.
	\end{split}
	\end{align}

\end{Def}

\subsection{Main Theorem:}

\begin{Th}\label{th1}
	Let pressure $p$ follows \eqref{p-condition}. We assume that the initial data is \emph{well-prepared}, i.e. it follows \eqref{well_prepared_id}. Let $ \vectorv_0\in W^{k,2}(\Om)$ with $k\geq 3$. Let $(\varrho_{\epsilon}, \vectoru_{\epsilon}) $ be a dissipative solution as in definition \eqref{dmv_defn} in $(0,T)\times \Om$.  Then,
	\begin{align}
	\begin{split}
	\text{ess}\sup_{t\in (0,T)} \Vert \varrho_{\epsilon} - \tilde{\varrho}_\epsilon \Vert_{(L^2+L^\gamma)(\Om)} \leq \epsilon^m c\\
	\frac{{\vectorm_\epsilon}}{\sqrt{\varrho_{\epsilon}}}\rightarrow \vectorv \begin{cases}
	 \text{weakly(*) in} L^{\infty}(0,T;L^2(\Om;\R^3)),\\
	\text{ strongly in } L^1_{\text{loc}}((0,T)\times \Om;\R^3),
	\end{cases}
	\end{split}
	\end{align}
	where, $\vectorv=(\vectorv_h,0)$ is the unique solution of Euler system with initial data $\vectorv_{0}$.
\end{Th}
In the following subsections we give the proof.
\subsection{Relative energy inequality}
In our approach, relative energy functional plays an important role. We consider,
\begin{align}\label{rel-ent-epsilon}
\begin{split}
\Epsilon_{\epsilon}(t)&=\Epsilon(\varrho_{\epsilon},\vectorm_{\epsilon} \vert \tilde{\varrho},\tilde{\vectoru})(t)\\
&:= \int_{\Omega}\big[ \frac{1	}{2} \varrho_{\epsilon} \bigg\vert \frac{\vectorm_\epsilon}{\varrho_{\epsilon}}-\tilde{\vectoru}\bigg\vert^2 + \frac{1}{\epsilon^{2m}}(P(\varrho_{\epsilon})-P(\tilde{\varrho}) -P^{\prime}(\tilde{\varrho})(\varrho_{\epsilon} -\tilde{\varrho})) \big] (t,\cdot) \dx ,
\end{split}
\end{align}
where $\tilde{\varrho},\tilde{\vectoru}$ are arbitrary smooth test functions with $\tilde{\varrho}-\tilde{\varrho}_\epsilon$ have compact support and  $\tilde{\vectoru}\cdot \mathbf{n}=0$ on $\partial\Om$.
\begin{Rem}
	The relative energy is a coercive functional (see. Bruell et. al. \cite{BrF2018}) satisfying the estimate,
	\begin{align}
	\begin{split}
	\Epsilon(\varrho_\epsilon,\vectoru_\epsilon \vert \tilde{\varrho},\tilde{\vectoru})(t) \geq& \int_{ \Om} \bigg[  \bigg\vert \frac{{\vectorm_\epsilon}}{\varrho_{\epsilon}}-\tilde{\vectorm}\bigg\vert^2  \bigg]_{\text{ess}} \dx + \int_{ \Om} \bigg[ \frac{ {\vert \vectorm_\epsilon \vert^2}} {\varrho_{\epsilon}} \bigg]_{\text{res}} \dx \\
	\quad & + \frac{1}{\epsilon^{2m}}\int_{ \Om} [(\varrho_\epsilon-\tilde{\varrho})^2]_{\text{ess}}\dx + \frac{1}{\epsilon^{2m}}\int_{ \Om} [1]_{\text{res}} + [\varrho_\epsilon^{\gamma}]_{\text{res}}\dx.
	\end{split}
	\end{align}
\end{Rem}
Using definition \eqref{dmv_defn} we obtain relative energy inequality,
\begin{align}\label{rel-ent-m}
\begin{split}
&\Epsilon_{\epsilon}(\tau)+   \int_{\overline{\Omega}}  {\rm d \; } \mathfrak{R}_{e_\epsilon} (\tau, \cdot)\\
\; &\leq \Epsilon_{\epsilon}(0) -\int_0^{\tau}\int_{\Om}  (\vectorm_{\epsilon}-\varrho_{\epsilon} \tilde{\vectoru})\cdot \partial_{t} \tilde{\vectoru}\dx \dt - \int_0^{\tau}\int_{\Om}\bigg(\frac{ (\vectorm_{\epsilon}-\varrho_{\epsilon} \tilde{\vectoru}) \otimes \vectorm_{\epsilon}}{\varrho_{\epsilon}}\bigg) : \nabla_x \tilde{\vectoru} \dx \dt  \\ 
&\;  - \frac{1}{\epsilon^{2m}} \int_0^{\tau}\int_{\Om} (p(\varrho_{\epsilon})-p(\tilde{\varrho})) \Div \tilde{\vectoru} \dx \dt +\frac{1}{\epsilon^{2m}} \int_0^{\tau}\int_{\Om} (\tilde{\varrho}-\varrho_{\epsilon}) \partial_{t} P^{\prime}(\tilde{\varrho}) \dx\dt \\
& \; +\frac{1}{\epsilon} \int_0^{\tau}\int_{\Om}  \mathbf{b} \times \mathbf{m}_{\epsilon} \cdot \bigg( \tilde{\vectoru}- \frac{{\vectorm_{\epsilon}}}{\varrho_{\epsilon}}\bigg)  \dx \dt \\ 
&\; + \frac{1}{\epsilon^{2m}} \int_0^{\tau}\int_{\Om}(\tilde{\varrho}\tilde{\vectoru}- \vectorm_{\epsilon}) \cdot ( \nabla_x P^{\prime}(\tilde{\varrho})-\nabla_x P^{\prime}(\tilde{\varrho_\epsilon})) \dx \dt \\
&\;  - \frac{1}{\epsilon^{2n}} \int_{0}^{\tau} \int_{ \Om} (\varrho_{\epsilon}- \tilde{\varrho}) \nabla_{x}G\cdot \tilde{\vectoru}\dx \dt- \int_0^{\tau}\int_{\Om} \nabla_x \tilde{\vectoru} : \text{d}\mathfrak{R}_{m_\epsilon}   (t,\cdot)\dt . \\
\end{split}
\end{align}

Following \cite{Bd2019} we can extend the above inequality for functions $(\tilde{\varrho},\tilde{\vectoru})$ having Sobolev regularities, i.e. $(\tilde{\varrho}-\tilde{\varrho}_\epsilon, \tilde{\vectoru})\in C^1([0,T]; W^{k,2}_{0}(\Om))\times C^1([0,T]; W^{k,2}(\Om;\R^3))$  with $k\geq 3$.

We rewrite the above inequality as,
\begin{align}
\begin{split}
\big[\Epsilon_{\epsilon}(\varrho_{\epsilon},\vectorm_{\epsilon}| \tilde{\varrho};\tilde{\vectoru})\big]_{0}^{\tau }+\int_{\overline{\Omega}}  {\rm d \; } \mathfrak{R}_{e_\epsilon} (\tau, \cdot)+ \big[\mathcal{R}_{\epsilon}(\varrho_{\epsilon},\vectorm_{\epsilon}| \tilde{\varrho};\tilde{\vectoru})\big] \leq 0.
\end{split}
\end{align}

Suppose, $(\tilde{\varrho}-\tilde{\varrho}_\epsilon, \tilde{\vectoru})\in C^1([0,T]; W^{k,2}_{0}(\Om))\times C^1([0,T]; W^{k,2}(\Om;\R^3))$  with $\tilde{\vectoru}\cdot \mathbf{n}|_{\partial\Om}=0$ and $k\geq3$. We know $C^1([0,T];C_{c}^{\infty}({ \Om}))$ is dense in Sobolev space $C^1([0,T]; W^{k,2}(\Om))$ and $\{\vectoru\in C^1([0,T];C^{\infty}(\overline{ \Om};\R^3))|\vectoru \cdot \mathbf{n}|_{\partial\Om=0} \}$ is dense in $\{\vectoru\in C^1([0,T];W^{k,2}({ \Om};\R^3))|\vectoru \cdot \mathbf{n}|_{\partial\Om=0} \}$. \\
For $\delta>0$ we have, $\tilde{r} \in C^1([0,T];C_{c}^{\infty}(\Om))$ and $\tilde{\vectorv} \in C^1([0,T];C^{\infty}(\overline{ \Om}))$ with $\tilde{\vectorv}\cdot \mathbf{n}|_{\partial\Om}=0$.
\begin{align*}
\Vert \tilde{r} - \tilde{\varrho} \Vert_{ C^1([0,T]; W^{k,2}(\Om))} + \Vert \tilde{\vectorv} - \tilde{\vectoru} \Vert_{ C^1([0,T]; W^{k,2}(\Om))} < \delta.
\end{align*}
Following Theorem 2.3 of \cite{Bd2019} we can show that,
\begin{align}
\begin{split}
& \big[\Epsilon_{\epsilon}(\varrho_{\epsilon},\vectorm_{\epsilon}| \tilde{\varrho};\tilde{\vectoru})\big]_{0}^{\tau } +\int_{\overline{\Omega}}  {\rm d \; } \mathfrak{R}_{e_\epsilon} (\tau, \cdot)+ \big[\mathcal{R}_{\epsilon}(\varrho_{\epsilon},\vectorm_{\epsilon}| \tilde{\varrho};\tilde{\vectoru})\big]\\
&\leq \big[\Epsilon_{\epsilon}(\varrho_{\epsilon},\vectorm_{\epsilon}| \tilde{r};\tilde{\vectorv})\big]_{0}^{\tau } +\int_{\overline{\Omega}}  {\rm d \; } \mathfrak{R}_{e_\epsilon} (\tau, \cdot)+\big[\mathcal{R}_{\epsilon}(\varrho_{\epsilon},\vectorm_{\epsilon}| \tilde{r};\tilde{\vectorv})\big] + C \delta \\
& \leq C\delta.\\
\end{split}
\end{align}
Thus for Sobolev functions, relative energy inequality \eqref{rel-ent-m} is true.

\subsection{Convergence: Part 1}First with $\tilde{\vectoru}= 0,\; \tilde{\varrho}=\tilde{\varrho}_\epsilon$ as test functions we have the following bounds,
\begin{align}\label{bound1}
\begin{split}
&\text{ess}\sup_{t\in (0,T)} \bigg\Vert \frac{\vectorm_{\epsilon}}{\sqrt{\varrho_{\epsilon}}} \bigg\Vert_{L^2(\Om;\R^3)} \leq C,\\
&\text{ess}\sup_{t\in (0,T)} \bigg\Vert \bigg[\frac{\varrho_{\epsilon}-\bar{\varrho}}{\epsilon^{m}}\bigg]_{\text{ess}}\bigg\Vert_{L^2(\Om)} \leq C,\\
&\text{ess}\sup_{t\in (0,T)} \Vert [\varrho_{\epsilon}]_{\text{res}} \Vert_{L^\gamma(\Om)}^{\gamma}  + \text{ess}\sup_{t\in (0,T)} \Vert [1]_{\text{res}} \Vert_{L^\gamma(\Om)}^{\gamma} \leq \epsilon^{2m} C .
\end{split}
\end{align}
Now we want to calculate $ \Vert \varrho_{\epsilon}- \tilde{\varrho}_\epsilon \Vert_{(L^2+L^\gamma)(\Om)}$. 
We rewrite,
\begin{align*}
 \Vert \varrho_{\epsilon}- \tilde{\varrho}_\epsilon \Vert_{(L^2+L^\gamma)(\Om)} \leq  \Vert [\varrho_{\epsilon}- \tilde{\varrho}_\epsilon]_{\text{ess}} \Vert_{L^2(\Om)} +  \Vert [\varrho_{\epsilon}- \tilde{\varrho}_\epsilon]_{\text{rss}} \Vert_{L^\gamma(\Om)}.
\end{align*}
From the above estimates and using fact $m>>\gamma$ we have,
\begin{align*}
\text{ess}\sup_{t\in (0,T)} \Vert \varrho_{\epsilon}- \tilde{\varrho}_\epsilon \Vert_{(L^2+L^\gamma)(\Om)} \leq \epsilon^m C.
\end{align*}
Now above estimates conclude that,
\begin{align}
\varrho_\epsilon \rightarrow 1 \text{ in } L^{\infty}(0,T;L^q_{\text{loc}}(\Om)) \text{ for any } 1\leq q <\gamma.
\end{align}
Combining above estimates we obtain,
\begin{align*}
& \frac{\vectorm_{\epsilon}}{\sqrt{\varrho_{\epsilon}}} \rightarrow \vectoru \text{ in } L^{\infty}(0,T;L^2(\Om;\R^3)),
\end{align*}
and
\begin{align*}
& \vectorm_{\epsilon}\rightarrow \vectoru \text{ weakly-}(*) \text{ in } {L^{\infty}(0,T;L^2+L^{\sfrac{2\gamma}{\gamma+1}}(\Om;\R^3))},
\end{align*}
passing to suitable subsequences.\\
Finally we may let $\epsilon \rightarrow 0$ in the continuity equation to deduce that,
\begin{align*}
\int_0^{\tau} \int_{\Om}  \vectoru \cdot \nabla_x \varphi \dx \dt=0,\; \forall \varphi \in C_{c}^{\infty}(\Om).
\end{align*}

\subsection{Convergence: Part 2}
Here we choose proper test functions and will show that $\lim\limits_{\epsilon\rightarrow 0} \Epsilon_{\epsilon}(t)=0$.\\
We consider,
\begin{align}
\tilde{\vectoru}= \vectorv,\; \tilde{\varrho}=\tilde{\varrho}_\epsilon+ \epsilon^m q_\epsilon,
\end{align}
where, $\vectorv=(\vectorv_h,0)$, $\vectorv_h$ as a solution of \eqref{target_sys_2D euler}, $q_\epsilon$ solves \eqref{non_osc} and $\tilde{\varrho}_\epsilon$ satisfies \eqref{static-solution}.
Using relation of $q_\epsilon$ and $\vectorv$ we obtain,
\begin{align}\label{rel_ent_m_2}
\begin{split}
&\Epsilon_{\epsilon}(\tau)+  \int_{\overline{\Omega}}  {\rm d \; } \mathfrak{R}_{e_\epsilon} (\tau, \cdot)\\
\; &\leq \Epsilon_{\epsilon}(0) -\int_0^{\tau}\int_{\Om}  (\vectorm_{\epsilon}-\varrho_{\epsilon} \vectorv)\cdot (\partial_{t}  \vectorv + ( \vectorv\cdot\nabla_{x}) \vectorv)\dx \dt \\
& - \int_0^{\tau}\int_{\Om}\bigg(\frac{ (\vectorm_{\epsilon}-\varrho_{\epsilon} {\vectorv}) \otimes (\vectorm_{\epsilon}-\varrho_{\epsilon} \vectorv)}{\varrho_{\epsilon}}\bigg) : \nabla_x \tilde{\vectoru} \dx \dt  \\ 
&\;   +\frac{1}{\epsilon^{2m}} \int_0^{\tau}\int_{\Om} (\tilde{\varrho}-\varrho_{\epsilon}) \partial_{t} P^{\prime}(\tilde{\varrho}) \dx\dt \\
& \; +\frac{1}{\epsilon^m} \int_0^{\tau}\int_{\Om}  \vectorm_{\epsilon} \cdot \nabla_{x}q_\epsilon (P^{\prime\prime}(\tilde{\varrho})-P^{\prime\prime}(1))  \dx \dt \\
&\;  + \frac{1}{\epsilon^{2m}} \int_0^{\tau}\int_{\Om} \vectorm_{\epsilon} \cdot  (P^{\prime\prime}(\tilde{\varrho})-P^{\prime\prime}(\tilde{\varrho_\epsilon})) \nabla_x \tilde{\varrho}_\epsilon \dx \dt \\
&\;  - \frac{1}{\epsilon^{2n}} \int_{0}^{\tau} \int_{ \Om} (\varrho_{\epsilon}- \tilde{\varrho}) \nabla_{x}G\cdot \tilde{\vectoru}\dx \dt- \int_0^{\tau}\int_{\Om} \nabla_x \tilde{\vectoru} : \text{d}\mathfrak{R}_{m_\epsilon}   (t,\cdot)\dt = \Sigma_{i=1}^{8} \mathcal{L}_{i}. \\
\end{split}
\end{align}

From the relation $
\nabla_{x} q_\epsilon + \epsilon^{m-1} \mathbf{b} \times (\vectorv_h,0)=0,$
and \eqref{reg_tar_system} we can conclude that
\begin{align}\label{q-epsilon est}
\Vert q_\epsilon\Vert_{L^{\infty}(0,T;L^q(\Om))} + \Vert \partial_{t} q_\epsilon \Vert_{L^{\infty}(0,T;L^q(\Om))} + \Vert  \nabla_{x_h}q_\epsilon \Vert_{L^{\infty}(0,T;L^q(\Om))} \leq \epsilon^{m-1}c,\; \forall q\geq 2.
\end{align}
Also we have $\Vert q_{0,\epsilon} \Vert_{L^{\infty}(0,T;L^2(\Om))} \leq \epsilon^{m-1}c$.\\
Let us calculate each term $\mathcal{L}_i$, $ i=1(1)8 $ of \eqref{rel_ent_m_2}. For term $\mathcal{L}_1$ we have,
\begin{align*}
\Epsilon_{\epsilon}(\varrho_{0,\epsilon}, (\varrho\vectoru)_{0,\epsilon} \;|\; \tilde{\varrho}_\epsilon+ \epsilon q_{0,\epsilon}, \vectorv_0) \leq \Vert \frac{(\varrho\vectoru)_{0,\epsilon}}{\varrho_{0,\epsilon}}-  \vectorv_{0} \Vert_{L^2(\Om)}^2+ \Vert \varrho_{0,\epsilon}^{(1)}-q_{0,\epsilon} \Vert^2_{L^2(\Om)}.
\end{align*}
Consideration of \emph{well prepared data} yields,
\begin{align}\label{estL1}
\vert \mathcal{L}_1 \vert \leq \xi(\epsilon).
\end{align}
From now on we use this generic function $\xi(\cdot)$, such that $\lim\limits_{\epsilon\rightarrow 0}\xi(\epsilon)=0$.\\
Now using our convergence in earlier part and $\vectorv_h$ solves \eqref{target_sys_2D euler} we obtain,
\begin{align*}
\mathcal{L}_2=& -\int_0^{\tau}\int_{\Om}  (\vectorm_{\epsilon}-\varrho_{\epsilon} \vectorv)\cdot (\partial_{t}  \vectorv + ( \vectorv\cdot\nabla_{x}) \vectorv)\dx \dt\\
&=-\int_0^{\tau}\int_{\Om}  (\vectorm_{\epsilon}- \vectorv)\cdot \nabla_{x_h} \Pi\dx \dt + -\int_0^{\tau}\int_{\Om}  (1-\varrho_{\epsilon})\vectorv\cdot \nabla_{x_h}\Pi \dx \dt\\
&= \mathcal{L}_{2\prime} + \mathcal{L}_{2\prime\prime}.
\end{align*}
It is easy to verify that,
\begin{align*}
\mathcal{L}_{2\prime} \xrightarrow{{\epsilon \rightarrow 0}} -\int_{0}^{\tau} \int_{ \Om} (\vectoru-\vectorv)\nabla_{x_h} \Pi\dx \dt=0.
\end{align*}
We also obtain,
$\vert \mathcal{L}_{2\prime\prime} \vert \leq \xi(\epsilon).$
Thus we can conclude, 
\begin{align}
\vert \mathcal{L}_2 \vert \leq \xi(\epsilon).
\end{align}
Now we also obtain,
\begin{align}
\vert \mathcal{L}_3 \vert \leq \Vert \nabla_{x_h} \vectorv_h \Vert_{L^{\infty}(0,\tau;L^{\infty}(\Om))} \int_{0}^{\tau} \Epsilon_{\epsilon}(t)\dt.
\end{align}
We want to estimate the term $\mathcal{L}_4$. First we rewrite it as,
\begin{align*}
\mathcal{L}_4&=\frac{1}{\epsilon^{m}} \int_0^{\tau}\int_{\Om} (\tilde{\varrho}-\varrho_{\epsilon}) P^{\prime\prime}(\tilde{\varrho}) \partial_{t} q_\epsilon  \dx\dt\\
&=\int_{0}^{\tau} \int_{ \Om} \frac{\tilde{\varrho}_\epsilon- \varrho_{\epsilon}}{\epsilon^m} \partial_{t} q_\epsilon P^{\prime\prime}(\tilde{\varrho})\dx \dt + \int_{0}^{\tau} \int_{ \Om} q_\epsilon \partial_{t} q_\epsilon P^{\prime\prime}(\tilde{\varrho}) \dx \dt. 
\end{align*}
We observe that,
\begin{align*}
\vert \mathcal{L}_4 \vert \leq &   \sup_{t\in (0,T)} \bigg\Vert \bigg[\frac{\varrho_{\epsilon}-\bar{\varrho}}{\epsilon^{m}}\bigg]_{\text{ess}}\bigg\Vert_{L^2(\Om)} \Vert \partial_t q_\epsilon \Vert_{L^{\infty}(0,T;L^2(\Om))}\\
&+ \sup_{t\in (0,T)} \bigg\Vert \bigg[\frac{\varrho_{\epsilon}-\bar{\varrho}}{\epsilon^{m}}\bigg]_{\text{res}}\bigg\Vert_{L^\gamma(\Om)} \Vert \partial_t q_\epsilon \Vert_{L^{\infty}(0,T;L^{\gamma\,'}(\Om))}\\
&+ \Vert \partial_t q_\epsilon \Vert_{L^{\infty}(0,T;L^2(\Om))} \Vert q_\epsilon \Vert_{L^{\infty}(0,T;L^2(\Om))}.
\end{align*}
Using bound of $q_\epsilon$ as in \eqref{q-epsilon est} and \eqref{bound1}, we obtain,
\begin{align}
\vert \mathcal{L}_4 \vert \leq \xi(\epsilon).
\end{align}
We write $\mathcal{L}_5$ as,
\begin{align*}
\mathcal{L}_5 &= \frac{1}{\epsilon^m} \int_{0}^\tau \int_{ \Om} \vectorm_{\epsilon} \cdot \nabla_{x} q_\epsilon (\tilde{\varrho}-1) P^{\prime\prime\prime}(\eta(x)) \dx \dt\\
&=\frac{1}{\epsilon^m} \int_{0}^\tau \int_{ \Om}\vectorm_{\epsilon} \cdot \nabla_{x} q_\epsilon (\tilde{\varrho}_\epsilon-1) P^{\prime\prime\prime}(\eta(x)) \dx \dt \\
&\quad\;\;+  \int_{0}^\tau \int_{ \Om} \vectorm_{\epsilon} \cdot \nabla_{x}  q_\epsilon q_\epsilon P^{\prime\prime\prime}(\eta(x)) \dx \dt.
\end{align*}
By using \eqref{bound1} and \eqref{q-epsilon est} we observe,
\begin{align}
\begin{split}
\vert \mathcal{L}_5 \vert &\leq  \epsilon^{m-2n} \Vert \vectorm_{\epsilon} \Vert_{L^{\infty}(0,T;L^2+L^{\sfrac{2\gamma}{\gamma+1}}(\Om;\R^3))} \Vert \nabla_{x} q_\epsilon \Vert_{L^{\infty}(0,T;L^2\cap L^{(\sfrac{2\gamma}{\gamma+1})^\prime}(\Om;\R^3))}\\
&\quad  + \Vert \vectorm_{\epsilon} \Vert_{L^{\infty}(0,T;L^2+L^{\sfrac{2\gamma}{\gamma+1}}(\Om;\R^3))}  \Vert q_\epsilon \nabla_{x} q_\epsilon \Vert_{L^{\infty}(0,T;L^2\cap L^{(\sfrac{2\gamma}{\gamma+1})^\prime}(\Om;\R^3))} \\
&\leq \xi(\epsilon).
\end{split}
\end{align}
Similarly for $\mathcal{L}_6$ we rewrite as,
\begin{align*}
\mathcal{L}_6 &= \frac{1}{\epsilon^{2m}} \int_{0}^\tau \int_{ \Om}  \vectorm_{\epsilon} \cdot (\tilde{\varrho} - \tilde{\varrho}_\epsilon) P^{\prime\prime\prime}(\zeta(x)) \nabla_{x} \tilde{\varrho}_\epsilon \dx \dt\\
& =\frac{1}{\epsilon^{m}} \int_{0}^\tau \int_{ \Om}  \vectorm_{\epsilon} \cdot  q_\epsilon P^{\prime\prime\prime}(\zeta(x)) \nabla_{x} \tilde{\varrho}_\epsilon \dx \dt.
\end{align*}
Arguing in the same line as before we have,
\begin{align}
\begin{split}
\vert \mathcal{L}_6 \vert &\leq \epsilon^{m-2n}  \Vert \vectorm_{\epsilon} \Vert_{L^{\infty}(0,T;L^2+L^{\sfrac{2\gamma}{\gamma+1}}(\Om;\R^3))} \Vert q_\epsilon  \Vert_{L^{\infty}(0,T;L^2\cap L^{(\sfrac{2\gamma}{\gamma+1})^\prime}(\Om))} \\
& \leq \xi(\epsilon) .
\end{split}
\end{align}
Now choice of $G$ implies
\begin{align}
\mathcal{L}_7 =0.
\end{align}   We also have,
\begin{align}\label{estL8}
\vert \mathcal{L}_8 \vert \leq \int_{0}^{\tau} \int_{\overline{ \Om}}\mathfrak{R}_{e_\epsilon} \dt .
\end{align}
 Thus combining all estimates \eqref{estL1}-\eqref{estL8} we have,
 \begin{align}
 \begin{split}
 \Epsilon_{\epsilon}(\tau)+ \int_{\overline{\Omega}}  {\rm d \; } \mathfrak{R}_{e_\epsilon} (\tau, \cdot) \leq \xi(\epsilon)+  c \int_{0}^{\tau} \Epsilon_{\epsilon} (t) \dt + c \int_{0}^{\tau} \int_{\overline{ \Om}}\text{d}\mathfrak{R}_{e_\epsilon}  \dt .
 \end{split}
 \end{align}
 Using Gr{\"o}nwall lemma we have,
 \begin{align}
 \begin{split}
 \Epsilon_{\epsilon}(\tau)+ \int_{\overline{\Omega}}  {\rm d \; } \mathfrak{R}_{e_\epsilon} (\tau, \cdot) \leq \xi(\epsilon) C(T),
 \end{split}
 \end{align}
 where $\xi(\epsilon)\rightarrow 0$ as $\epsilon \rightarrow 0$.
Using coercivity of relative energy functional we obtain,
\begin{align*}
\limsup_{\epsilon\rightarrow 0} \int_{K} \vert \frac{{\vectorm_\epsilon}}{\sqrt{\varrho_{\epsilon}}}-\vectorv \vert^2 \dx \leq C(T) \limsup_{\epsilon\rightarrow  0} \xi(\epsilon),
\end{align*}
where, $K\subset \Om$ is a compact set.
Hence we can conclude that $\vectorv=\vectoru=\vectorv_h$. Also we have,
\begin{align*}
\frac{{\vectorm_\epsilon}}{\sqrt{\varrho_{\epsilon}}}\rightarrow \vectorv \text{ \emph{strongly} in }L^1_{\text{loc}}((0,T)\times \Om;\R^3).
\end{align*}

It ends proof of the theorem.

\section{{Extension to the Navier--Stokes System}}

Here our goal is to give a proper definition of \emph{dissipative} solution for Navier--Stokes equation. We consider another characteristic number i.e. Reynolds number. In high Reynolds number limit, we will obtain the same \emph{ target system}.
\subsection{Definition of dissipative Solution for Navier-Stokes system :}
Let $\varrho$ be the density and $\vectoru$ be the velocity. In time-space cylinder $Q_T=(0,T)\times \Omega$, we consider: 
\begin{itemize}
	\item \textbf{Conservation of Mass: }
	\begin{align}
	\partial_t \varrho + \text{div}_x (\varrho\vectoru)&=0. \label{cnse:cont}
	\end{align}
	\item \textbf{Conservation of Momentum:}
	\begin{align}\label{cnse:mom:sc}
	\begin{split}
	\partial_t(\varrho \mathbf{u}) + \Div (\varrho \mathbf{u} \otimes \mathbf{u})+\frac{1}{\text{Ma}^2}\nabla_x p(\varrho)+\frac{1}{\text{Ro}} \mathbf{b} \times \varrho\mathbf{u}\\
	=\frac{1}{\text{Re}}\Div \mathbb{S}(\nabla_x \mathbf{u})+ \frac{1}{\text{Fr}^2}\varrho \nabla_{x}G. 
	\end{split}
	\end{align}	
	
	\item  \textbf{Constitutive Relation:}
	Here $\mathbb{S}(\nabla_x \mathbf{u})$ is \textit{Newtonian stress tensor} defined by
	\begin{align}\label{cauchy_str}
	\mathbb{S}(\nabla_x \mathbf{u})=\mu \bigg(\frac{\nabla_x \vectoru + \nabla_x^{T} \vectoru}{2}-\frac{1}{d} (\Div\vectoru)\mathbb{I} \bigg) + \lambda (\Div \vectoru) \mathbb{I},
	\end{align}
	where $\mu >0$ and $\lambda >0$ are the \textit{shear} and \textit{bulk} viscosity coefficients, respectively. 
\end{itemize}
\begin{itemize}
	\item The scaled system contains all specified  characteristic numbers as in \eqref{scaling} along with,
	\subitem Re-- Reynolds number.\\
	Here  we consider,
	\begin{align}\label{scaling:nse}
	\text{Ma} \approx \epsilon^m,\; \text{Ro} \approx \epsilon,\; \text{Re}\approx \epsilon^{-\alpha},\; \text{Fr} \approx \epsilon^n \text{ for } \epsilon>0,\; m,n,\alpha>0 \text{ and } \frac{m}{2}>n\geq1.
	\end{align}
	
	\item \textbf{Pressure Law:} In an isentropic setting, the pressure $p$ and the density $\varrho$ of the fluid are interrelated by
	\begin{align}\label{p-condition:nse}
	\begin{split}
	p(\varrho)=a \varrho^\gamma,\; a>0,\; \gamma> 1.
	\end{split}
	\end{align}
	
	\item \textbf{Boundary condition:} Here we consider complete slip condition for velocity on the horizontal boundary i.e.
	\begin{align}\label{BC:nse}
	\vectoru \cdot \mathbf{n}|_{\partial\Om}=[\mathbb{S}(\nabla_{x}\vectoru)\cdot \mathbf{n}]_{\text{tan}}|_{\partial\Om}=0,\; \mathbf{n}=(0,0,\pm 1).
	\end{align}
	
	\item \textbf{Far field condition:} Let $(\tilde{\varrho}_{\epsilon},\mathbf{0})$ be a static solution. 
	we assume the condition as,
	\begin{align}\label{far_field:nse}
	|\varrho - \tilde{\varrho}_\epsilon| \rightarrow 0, \; \vectoru \rightarrow \mathbf{0} \text{ as } \vert x_h \vert \rightarrow \infty.
	\end{align}
	\item \textbf{Initial data:}
	For each $\epsilon>0$, we supplement the initial data as
	\begin{align}\label{initial condition:nse}
	{\varrho(0,\cdot)=\varrho_{\epsilon, 0},\; (\varrho\vectoru)(0,\cdot)= (\varrho \vectoru)_{\epsilon, 0}.}
	\end{align}
	
\end{itemize}
 Here we provide the definition of \emph{dissipative solution} for the Navier--Stokes system.

\begin{Def}\label{diss_defn_NS}
	Let $\epsilon>0$ and $\tilde{\varrho}_\epsilon>0$. We say functions $\varrho_{\epsilon},\vectoru_{\epsilon}$  with,
	\begin{align}
	\varrho_{\epsilon}-\tilde{\varrho_\epsilon} \in  &C_{\text{weak}}([0,T];L^2+L^{\gamma}(\Om)),\;\varrho_\epsilon \geq 0,\; \varrho_{\epsilon}\vectoru_{\epsilon} \in   C_{\text{weak}}([0,T];L^2+L^{\frac{2\gamma}{\gamma+1}}(\Om)),\\
	&\vectoru_{\epsilon}\in L^2(0,T;W^{1,2}(\Om)),\; \vectoru_{\epsilon}\cdot \mathbf{n}|_{\partial\Om}=0,
	\end{align}
	are a \emph{dissipative solution} to \eqref{cnse:cont}-\eqref{far_field:nse} with initial data $\varrho_{0,\epsilon}, (\varrho \vectoru)_{0,\epsilon}$ satisfying,
	\begin{align}\label{fe_id:nse}
	\begin{split}
	\varrho_{0,{\epsilon}}\geq 0,\; E_{0,{\epsilon}}=\int_{ \Om} \bigg( \frac{1}{2} \frac{ \vert (\varrho \vectoru)_{0,\epsilon} \vert^2}{\varrho_{0,\epsilon}} + P(\varrho_{0,{\epsilon}})-(\varrho_{0,{\epsilon}}-\tilde{\varrho}_\epsilon)P^{\prime}(\tilde{\varrho}_\epsilon)-P(\tilde{\varrho}_\epsilon)\bigg)\dx < \infty,
	\end{split}
	\end{align}
	if there exist the \emph{turbulent defect measures}
	\begin{align}
	\begin{split}
	&{\mathfrak{R}_{m_{\epsilon}} \in L^{\infty}(0,T;\mathcal{M}^+({\overline{\Om};\mathbb{R}^{d\times d}_{\text{sym}}}))},\; \mathfrak{R}_{e_{\epsilon}} \in L^{\infty}(0,T;\mathcal{M}^{+}(\overline{\Om})),
	\end{split}
	\end{align}
	satisfying compatibility condition
	\begin{align}\label{compat_turb_mes_NS}
	\lambda_1 \text{trace}\mathfrak{R}_{m_{\epsilon}} \leq \mathfrak{R}_{e_{\epsilon}} \leq \lambda_2 \text{trace} \mathfrak{R}_{m_{\epsilon}},\; \lambda_1,\lambda_2>0,
	\end{align}
	such that the following holds,
	\begin{itemize}
		\item \textbf{Equation of Continuity:}
		For any $\tau \in (0,T) $ and any $\varphi \in C_{c}^{1}([0,T]\times \bar{\Om})$ {it holds}
		\begin{align}\label{continuity_eqn_weak}
		\begin{split}
		&\big[ \int_{ \Om}{ \varrho_{\epsilon}} \varphi \dx\big]_{t=0}^{t=\tau}=
		\int_0^{\tau} \int_{\Om} [ \varrho_{\epsilon} \partial_t \varphi +  \varrho_{\epsilon}\vectoru_{\epsilon} \cdot \nabla_x \varphi] \dx \dt .
		\end{split}
		\end{align}
		
		\item \textbf{Momentum equation:} For any $\tau\in (0,T)$ and any $\pmb{\varphi} \in C^{1}_c([0,T]\times \Om;\mathbb{R}^d)$ with $ \pmb{\varphi}\cdot \mathbf{n} |_{\partial{\Om}}=0$, it holds
		\begin{align}\label{momentum_eqn_weak}
		\begin{split}
		&\bigg[\int_{\Om}  \varrho_{\epsilon}\vectoru_{\epsilon}(\tau,\cdot)\cdot \vectorphi(\tau,\cdot) \dx\bigg]_{t=0}^{t=\tau} \\
		&=\int_0^{\tau}\int_{\Om} [ \varrho_{\epsilon}\vectoru_{\epsilon}\cdot \partial_{t} \vectorphi + \varrho_{\epsilon}\vectoru_{\epsilon}\otimes\vectoru_{\epsilon} : \nabla_x \vectorphi + \frac{1}{\epsilon^{2m}}p(\varrho_{\epsilon}) \Div \vectorphi + \frac{1}{\epsilon}\mathbf{b} \times \varrho_{\epsilon}\vectoru_{\epsilon} \cdot \vectorphi ] \dx \dt \\
		&-\int_0^{\tau}\int_{\Om}[\epsilon^\alpha\mathbb{S}(\nabla_x \vectoru):\nabla_x \vectorphi-  \frac{1}{\epsilon^{2n}} \varrho_{\epsilon} \nabla_{x}G\cdot \vectorphi] \dx \dt+\int_0^{\tau}\int_{\overline{\Om}} \nabla_x \vectorphi : \text{d}\mathfrak{R}_{m_{\epsilon}} \dt.
		\end{split}
		\end{align}
		\item \textbf{Energy inequality:} The total energy $E$ is defined in $[0,T)$ as,
		\begin{align*}
		E_{\epsilon}(\tau)= \int_{\Om} &\bigg( \frac{1}{2} {\varrho_\epsilon}{ {\vert \vectoru_{\epsilon} \vert^2}} +\frac{1}{\epsilon^{2m}}(P(\varrho_{\epsilon})-(\varrho_{\epsilon}-\tilde{\varrho}_\epsilon)P^{\prime}(\tilde{\varrho}_{\epsilon})) -P(\tilde{\varrho}_{\epsilon})\bigg)(\tau ,\cdot)\dx.
		\end{align*}
		It satisfies,

		\begin{align}\label{energy_inequality:nse}
		E_\epsilon (\tau) +\epsilon^\alpha  \int_{0}^{\tau} \int_{ \Om} \mathbb{S}(\nabla_x \vectoru) :\nabla_x \vectoru\dx \dt+ \int_{\overline{\Omega}}  {\rm d \; } \mathfrak{R}_{e_\epsilon} (\tau, \cdot) \leq E_{0,\epsilon}
		\end{align}
		{for a.a. $\tau > 0$.}			
	\end{itemize}
\end{Def}

\begin{Rem}
		The class of test functions in the momentum equations correspond to the 
		complete slip (Navier slip) boundary conditions. These are necessary to avoid problems with boundary layer. 
\end{Rem}

\begin{Th}
	Suppose $\Om$ be the domain specified above and pressure follows \eqref{p-condition:nse}. If $(\varrho_{0,\epsilon},(\varrho \vectoru)_{0,\epsilon})$ satisfies \eqref{fe_id:nse}, then there exists \emph{dissipative solution} as defined above.
\end{Th}
We prove the existence theorem for $\epsilon=1$.
\begin{proof}
	Here we give an extended outline of proof. We know that the existence theory in the class of finite energy weak solutions was developed by Lions \cite{PL1998} and later extended by Feireisl \cite{F2004b} to the so far subcritical exponent $\gamma>\frac{3}{2}$. For unbounded domain similar result has been proposed by Novotn{\'y} and Pokorn{\'y} in \cite{NP2007}.\par
	Here our goal is to add $\delta \nabla_{x} \varrho^\Gamma $ in the momentum equation with $\Gamma\geq \frac{3}{2}$ and then the system admits a finite energy weak solutions $(\varrho_\delta,\vectoru_\delta)$. Then we will show that this approximate solution converges to a dissipative solution of above described system.
	%From Feireisl \cite{F2002}, we have existence of weak solution for large adiabatic exponent $\gamma$ . 
	This motivates the following approximate problem,
	\begin{align}
	\partial_t \varrho_{\delta} + \text{div}_x (\varrho_{\delta} \mathbf{u}_{\delta})&=0, \label{cnse:cont:ap}\\\
	\partial_t(\varrho_{\delta} \mathbf{u}_{\delta}) + \Div (\varrho_{\delta} \mathbf{u}_{\delta} \otimes \mathbf{u}_{\delta})+\nabla_x p(\varrho_{\delta})+ &{\delta} \nabla_{x} \varrho_{{\delta}}^{\Gamma}+ \mathbf{b}\times \varrho_{\delta} \vectoru_{\delta} \nonumber \\
	& =\Div \mathbb{S}(\nabla_x \mathbf{u}_{\delta}) + \varrho_{{\delta}}\nabla_{x}G, \label{cnse:mom:ap}\\
	\vectoru_{\delta}\cdot \mathbf{n}|_{\partial \Om}&=[\mathbb{S}(\nabla_{x}\vectoru_\delta)\cdot \mathbf{n}]_{\text{tan}}|_{\partial\Om}=0 . \label{bc:ap}
	\end{align}
	with $\Gamma>\frac{3}{2}$.
	We assume that for each $\delta>0$ the \emph{static solution} of approximate problem $\tilde{\varrho}_\delta$ has the following property 
	\begin{align}
	\sup_{  x_3\in [0,1]}\vert \tilde{\varrho}_\delta - \tilde{\varrho}\vert \leq \delta,
	\end{align}
	where, $\tilde{\varrho}>0$ is static solution for the Navier--Stokes problem with $\nabla_{x} P^{\prime} (\tilde{\varrho})=\nabla_{x}G$.
	Further we assume that for the above mentioned problem, initial condition $\{ \varrho_{\delta,0}, (\varrho \vectoru)_{\delta,0} \}$ belongs to a certain regularity class for which weak solution exists. As an additional assuption we have,
	\begin{align*}
	 E^{\delta}_{0,{\delta}}=\int_{ \Om} \bigg( \frac{1}{2} \frac{ \vert (\varrho \vectoru)_{0,\delta} \vert^2}{\varrho_{0,\delta}} + H(\varrho_{0,{\delta}})-(\varrho_{0,{\delta}}-\tilde{\varrho}_\delta)H^{\prime}(\tilde{\varrho}_\delta)-H(\tilde{\varrho}_\delta)\bigg)\dx \\
	 \rightarrow	 E_{0}=\int_{ \Om} \bigg( \frac{1}{2} \frac{ \vert (\varrho \vectoru)_{0} \vert^2}{\varrho_{0}} + P(\varrho_{0})-(\varrho_{0}-\tilde{\varrho})P^{\prime}(\tilde{\varrho})-P(\tilde{\varrho})\bigg)\dx\text{ in } L^1(\Om),
	\end{align*}
	where $H(s)=\frac{1}{\gamma-1}p(s)+\delta \frac{1}{\Gamma-1} s^{\Gamma}$.\\	
	Clearly from existence of weak solution we have several apriori bounds, i.e.
	\begin{align*}
	&\Vert \varrho_\delta  - \tilde{\varrho}_\delta \Vert_{L^{\infty}(0,T;L^2+L^{\gamma}(\Om))} \leq C,\\
	& \Vert \sqrt{\varrho_{{\delta}}} \vectoru_{\delta} \Vert_{L^2(\Om;\R^3)} \leq C,\\
	& \Vert \vectoru_\delta \Vert_{L^2(0,T;W^{1,2}(\Om:\R^3)}\leq C,\\
	&\Vert H(\varrho_{{\delta}})-(\varrho_{{\delta}}-\tilde{\varrho}_\delta)H^{\prime}(\tilde{\varrho}_\delta)-H(\tilde{\varrho}_\delta) \Vert_{L^{\infty}(0,T;L^1(\Om))} \leq C.
	\end{align*}
	As a consequence of that we obtain,
	\begin{align}
	\begin{split}
	&\Vert \varrho_{\delta} \vectoru_{\delta} \Vert_{L^{\infty}(0,T;L^2+L^{\sfrac{2\gamma}{\gamma+1}}(\Om;\R^3))} \leq C,\\
	& \text{ess}\sup_{  t\in [0,T]} \int_{ \Om} \delta \varrho_{\delta}^\Gamma \dx \leq C	.
	\end{split}
	\end{align}
	From the above bounds we get,
	\begin{align*}
	&\varrho_\delta - \tilde{\varrho}_\delta \rightarrow \varrho -\tilde{\varrho} \text{ weakly(*) in } {L^{\infty}(0,T;L^2+L^{\gamma}(\Om))},\\
	&\vectoru_\delta \rightarrow \vectoru \text{ weakly in } {L^2(0,T;W^{1,2}(\Om:\R^3)}.
	\end{align*}
	Following the above consequence we also conclude,
	\begin{align*}
	\varrho_\delta \vectoru_\delta \rightarrow \varrho \vectoru \text{ weakly(*) in } {L^{\infty}(0,T;L^2+L^{\sfrac{2\gamma}{\gamma+1}}(\Om))}.
	\end{align*}
	Let us introduce the conservative variable $\vectorm_\delta=\varrho_{{\delta}}\vectoru_{\delta}$.
	In terms of momentum we rewrite kinetic energy as,
	\[ (\varrho_{{\delta}},\vectorm_\delta )\mapsto \frac{1}{2}\frac{\vert \vectorm_\delta\vert^2}{\varrho_\delta} =
	\begin{cases} 
	\frac{1}{2}\frac{\vert (\varrho_{{\delta}}\vectoru_\delta)\vert^2}{\varrho_\delta}  & \text{ if } ,\varrho_\delta \neq 0, \vectorm_\delta \neq 0 ,\\
	0 & \text{ if } \varrho_\delta =0,\; \vectorm_\delta=0 ,\\
	\infty & \text{ if }\varrho_\delta=0,\; \vectorm_\delta\neq 0.
	\end{cases}
	\]
	As an observation we have the above map is convex l.s.c. From energy inequality, it is worth to notice that it is $\infty$ only on a measure zero set in $(0,T)\times \Om$.
	Using convexity of $p(\cdot)$ and $[\varrho,\vectorm]\mapsto \frac{\vectorm \times \vectorm}{\varrho}$, and also using fact $L^1(\Om)$ continuously embedded in $\mathcal{M}(\bar{\Om})$, we conclude,
	
	\begin{align}
	\begin{split}
	& \frac{\vectorm_\delta \times \vectorm_\delta}{\varrho_\delta}\rightarrow \overline{\frac{\vectorm \times \vectorm}{\varrho}} \text{ weakly-(*) in } L^{\infty}(0,T;\mathcal{M}(\overline{\Om};\R^{d\times d}_{\text{sym}})),\\
	& p(\varrho_\delta)  \rightarrow \overline{p(\varrho)} \text{ weakly-(*) in } L^{\infty}(0,T;\mathcal{M}(\overline{\Om})),\\
	&\delta \varrho_\delta^\Gamma (t,\cdot) \rightarrow \zeta \text{ weakly-(*) in } L^{\infty}(0,T;\mathcal{M}(\overline{\Om})).
	\end{split}
	\end{align}
	We choose,
	\begin{align}
	&\mathfrak{R}_{m}= \bigg[\overline{\frac{\vectorm \times \vectorm}{\varrho}}-{\frac{\vectorm \times \vectorm}{\varrho}}\bigg]+\big[ \overline{p(\varrho)}-p(\varrho)+ \zeta\big]\mathbb{I},
	\end{align}
	and
	\begin{align}
	\mathfrak{R}_e= \overline{\frac{1}{2}\frac{\vert \vectorm \vert^2}{\varrho}}-\frac{1}{2}\frac{\vert \vectorm \vert^2}{\varrho} + \overline{P(\varrho)}-P(\varrho)+ \frac{1}{\Gamma-1}\zeta .
	\end{align}

	Clearly the compatibility of two \emph{turbulent defect measure} is clear from above equations. Arguing similarly as in Breit et. al. \cite{BeFH2019(2)} we obtain,
	\[ {\mathfrak{R}_{m} \in L^{\infty}(0,T;\mathcal{M}^+({\overline{\Om};\mathbb{R}^{d\times d}_{\text{sym}}}))},\; \mathfrak{R}_{e} \in L^{\infty}(0,T;\mathcal{M}^{+}(\overline{\Om})).\] Now we are in a position to conclude that $\varrho, \vectoru, \mathfrak{R}_m \text{ and } \mathfrak{R}_e$ is a \emph{dissipative solution }for the Navier--Stokes equation.
	
\end{proof}

Finally we state the theorem,
\begin{Th}\label{th2}
	Let pressure $p$ follows \eqref{p-condition:nse}. We assume that the initial data is \emph{well-prepared}, i.e. it follows \eqref{well_prepared_id}. We also consider $\vectorv_0\in W^{k,2}$ with $k\geq 3$. Let $(\varrho_{\epsilon}, \vectoru_{\epsilon}) $ be a dissipative solution as in definition \eqref{diss_defn_NS} in $(0,T)\times \Om$.  Then,
	\begin{align}
	\begin{split}
	&\text{ess}\sup_{t\in (0,T)} \Vert \varrho_{\epsilon} - \tilde{\varrho}_\epsilon \Vert_{(L^2+L^\gamma)(\Om)} \leq \epsilon^m c,\\
	&{\sqrt{\varrho_{\epsilon}}}{\vectoru_{\epsilon}} \rightarrow \vectorv \begin{cases}
	\text{weakly(*) in} L^{\infty}(0,T;L^2(\Om;\R^3)),\\
	\text{ strongly in } L^1_{\text{loc}}((0,T)\times \Om;\R^3),
	\end{cases}
	\end{split}
	\end{align}
	where, $\vectorv=(\vectorv_h,0)$ and $\vectorv_h$ is the unique solution of Euler system \eqref{target_sys_2D euler} with initial data $\vectorv_{0}$.
\end{Th}
\begin{proof}
	The proof is similar as before only we have to consider few extra terms, see \cite{FN2014ii}.
\end{proof}

\vspace{5mm}

\centerline{ \bf Acknowledgement}
\vspace{2mm}

The work is supported by Einstein Stiftung, Berlin. I would like to thank my  Ph.D supervisor Prof. E. Feireisl for his valuable suggestions and comments.

\end{document}